\documentclass{article}
\usepackage{amsmath, amsthm, latexsym, amssymb, mathabx, color,cite,enumerate, physics, framed}
\usepackage{caption,subcaption,verbatim, empheq, cancel, esint, bbm,multicol,hyperref}
\newtheorem{theorem}{Theorem}
\newtheorem{lemma}{Lemma}

\author{Evan Randles}
\title{A note on the completeness of Fourier-based metrics on measures}
\date{}

\begin{document}
\maketitle
\begin{abstract}Resolving an open question of J. A. Carillo and G. Toscani, M. Stawiska recently proved that certain metric spaces of probability measures equipped with Fourier-based metrics are complete. In this note, we extend such Fourier-based metrics to the class of complex Borel measures and consider the question of completeness in this broader setting.
\end{abstract}


\\

\noindent We denote by $\mathcal{M}=\mathcal{M}(\mathbb{R}^d)$ the set of complex-valued finite Borel measures on $\mathbb{R}^d$. For each $\mu\in \mathcal{M}$, its Fourier transform $\widehat{\mu}:\mathbb{R}^d\to\mathbb{C}$ is defined by
\begin{equation*}
\widehat{\mu}(\xi)=\int_{\mathbb{R}^d}e^{-ix\cdot\xi}\,d\mu(x)
\end{equation*}
for $\xi\in\mathbb{R}^d$; here, $\cdot$ denotes the dot product. It is a standard fact that measures are characterized by their Fourier transforms, i.e., for $\mu,\nu\in\mathcal{M}$, $\mu=\nu$ if and only if $\widehat{\mu}=\widehat{\nu}$ \cite[Theorem 1.11]{StrombergBook}. Let us fix an integer $m\geq 2$ and, for a pair of measures $\mu,\nu\in \mathcal{M}$, set
\begin{equation*}
d_m(\mu,\nu)=\sup_{\xi\neq 0}\frac{\abs{\widehat{\mu}(\xi)-\widehat{\nu}(\xi)}}{\abs{\xi}^m}
\end{equation*}
where $\abs{\cdot}$ denotes the Euclidean norm on $\mathbb{R}^d$. It is straightforward to see that $d_m$ is symmetric, non-negative, and satisfies the triangle inequality. By virtue of the uniqueness property of the Fourier transform mentioned above, we see that $d_m(\mu,\nu)=0$ if and only if $\mu=\nu$. Thus $d_m$ is a metric on any subset of $\mathcal{M}$ for which it is finite. 

The metric $d_m$, when restricted to the space of probability measures $\mathcal{P}(\mathbb{R}^d)\subseteq\mathcal{M}$, was introduced in the study of asymptotics of solutions to the Boltzmann equation \cite{GTW95}. A general theory for such metrics and their applications to statistical physics can be found in \cite{CT07}. Addressing an open question of J. A. Carillo and G. Toscani (posed in \cite{CT07}), M. Stawiska recently sorted out the question of completeness for spaces of probability measures equipped with the metric $d_m$. This result appears as Theorem 1.1 of \cite{Sta20} and is described precisely below. In the spirit of the works \cite{BD93, DSC14, RSC15, RSC17}, we are interested in various questions and phenomena with well-known analogues in probability but posed in the broader complex-valued setting. One goal, in particular, is to understand the extent to which certain results depend on positivity. With this as partial motivation, the purpose of this note is to generalize the metric spaces of \cite{GTW95, CT07,Sta20} (equipped with $d_m$) to the complex setting and study the analogous question of completeness.

To set the stage, we first construct subsets of $\mathcal{M}$ on which $d_m$ becomes a metric.  As pointed out in \cite{CT07} and \cite{Sta20}, key to this construction is the fact that $d_m$ behaves nicely on subsets of measures with (finite and) identical moments. Denote by $\mathcal{M}_m$ the subset of $\mathcal{M}$ which have finite moments of order $m$, i.e., $\mu\in \mathcal{M}_m$ if and only if 
\begin{equation*}
\int_{\mathbb{R}^d}\abs{x}^m\,d\mu(x)<\infty. 
\end{equation*}
Of course, because $1\in L^1(\mu)$, each $\mu\in\mathcal{M}_m$ has finite moments of order up to and including $m$. For a specific choice $M=\{M_\beta\}_{\abs{\beta}\leq m}\subseteq\mathbb{C}$ of moments, let
\begin{equation*}
\mathcal{M}_m^M=\left\{\mu\in \mathcal{M}_m \left\vert \, \,\int_{\mathbb{R}^d} x^\beta\,d\mu(x)=M_\beta\,\,\mbox{for}\,\,\abs{\beta}\leq m\right\}\right.;
\end{equation*}
here, for a multi-index $\beta=(\beta_1,\beta_2,\dots,\beta_d)\in\mathbb{N}^d$, $\abs{\beta}=\beta_1+\beta_2+\cdots+\beta_d$ and, for $x=(x_1,x_2,\dots,x_d)\in\mathbb{R}^d$, $x^\beta=x_1^{\beta_1}\cdot x_2^{\beta_2}\cdots x_d^{\beta_d}$. The set $\mathcal{M}_m^M$ is precisely the set of measures $\mu\in \mathcal{M}_m$ with identical moments prescribed by $M=\{M_\beta\}_{\abs{\beta}\leq m}$. When one considers the analogous (sub) class of probability measures $\mathcal{P}_m^M\subseteq\mathcal{M}_m^M$ (where, necessarily, $\{M_\beta\}\subseteq\mathbb{R}$), Theorem 1.1 of \cite{Sta20} asserts that $(\mathcal{P}_m^M(\mathbb{R}^d),d_m)$ is complete when $m$ is even and, for certain choices of $M$, not complete when $m$ is odd. Our main result is that, in the general context of complex Borel measures, every integer $m\geq 2$ and choice of moments $M$ give rise to a non-complete metric space. Precisely, this is the following theorem.



\begin{theorem}\label{thm:Metric_Not_Compete}
Given any integer $m\geq 2$ and $M=\{M_\beta\}_{\abs{\beta}\leq m}\subseteq\mathbb{C}$, $(\mathcal{M}_m^M, d_m)$ is a metric space which is not complete.
\end{theorem}

\noindent In order the prove the theorem, we first present four lemmas. The first shows that the Fourier transform exchanges the finiteness of moments for smoothness; its proof is standard and omitted.
\begin{lemma}\label{lem:MomentsAndSmoothness}
If $\mu\in \mathcal{M}_m$, then $\widehat{\mu}$ is $m$-times continuously differentiable on $\mathbb{R}^d$, i.e., $\widehat{\mu}\in C^m(\mathbb{R}^d)$ and, for each multi-index $\beta\in\mathbb{N}^d$ for which $\abs{\beta}\leq m$, 
\begin{equation*}
(D^\beta\widehat{\mu})(\xi)=(-i)^{\abs{\beta}}\int_{\mathbb{R}^d}x^\beta e^{-ix\cdot\xi}\,d\mu(\xi)
\end{equation*}
for $\xi\in\mathbb{R}^d$; here, for a multi-index $\beta=(\beta_1,\beta_2,\dots,\beta_d)\in\mathbb{N}^d$, $D^\beta=\partial_{\xi_1}^{\beta_1}\partial_{\xi_2}^{\beta_2}\cdots\partial_{\xi_d}^{\beta_d}$.
\end{lemma}

\begin{lemma}\label{lem:SchwartzFunctions}
Let $m$ and $\{M_\beta\}$ be as in the statement of the theorem and fixed. Consider the multivariate polynomial
\begin{equation*}
P(\xi)=\sum_{|\gamma|\leq m} \frac{M_\gamma}{\gamma!}(-i\xi)^\gamma
\end{equation*}
where $\gamma!=(\gamma_1!)(\gamma_2!)\cdots(\gamma_d!)$ and, for $\delta\geq 0$, set
\begin{equation*}
\phi_\delta(\xi)=\left(P(\xi)+\abs{\xi}^{m+2}\sin\left(\frac{1}{\abs{\xi}^m}\right)e^{-\delta/|\xi|^2}\right)e^{-\abs{\xi}^{2m}}
\end{equation*}
for $\xi\in\mathbb{R}^d\setminus\{0\}$ and $\phi_\delta(0)=P(0)$. If $\delta>0$, then $\phi_\delta$ is the Fourier transform of a measure $\mu_{1/\delta}\in\mathcal{M}_m^M$.
\end{lemma}

\begin{proof}
For $\delta>0$, define
\begin{equation*}
f_\delta(\xi)=\begin{cases}
\abs{\xi}^{m+2}\sin(1/\abs{\xi}^m)e^{-\delta/\abs{\xi}^2} & \xi\neq 0\\
0 &\xi=0
\end{cases}
\end{equation*}
so that $\phi_\delta(\xi)=(P(\xi)+f_{\delta}(\xi))e^{-\abs{\xi}^{2m}}$ for $\xi\in\mathbb{R}^d$. Observe that $f_\delta$ is smooth on $\mathbb{R}^d\setminus\{0\}$. Further, for each multi-index $\beta\in\mathbb{N}^d$, it is straightforward to see that there are multivariate polynomials $Q_1$ and $Q_2$ (in the variables $\xi_1,\xi_2,\dots,\xi_d,$ and $\abs{\xi}$) and an integer $k\geq 0$ for which
\begin{equation}\label{eq:Derivative}
(D^\beta f_\delta)(\xi)=\left(Q_1(\xi)\sin(1/\abs{\xi}^m)+Q_2(\xi)\cos(1/\abs{\xi}^m)\right)\frac{e^{-\delta/\abs{\xi}^2}}{\abs{\xi}^k}
\end{equation}
for $\xi\neq 0$. We now prove that, for every $\beta\in\mathbb{N}^d$, $(D^\beta f_\delta)$ is continuous at $0$ with $(D^\beta f_\delta)(0)=0$ by inducting along chains. The base case, i.e., that for $\beta=0$, is immediate. Assuming the induction hypothesis, observe that, for $1\leq j\leq d$, 
\begin{eqnarray*}
(D^{\beta+e_j}f_\delta)(0)&=&\lim_{h\to 0}\frac{(D^{\beta}f_{\delta})(he_j)-0}{h}\\
&=&\lim_{h\to 0}\left(Q_1(he_j)\sin(1/\abs{h}^m)+Q_2(he_j)\cos(1/\abs{h}^m)\right)\frac{e^{-\delta/h^2}}{h^{k+1}}=0
\end{eqnarray*}
on account of the fact that $e^{-\delta/h^2}=o(h^{k+1})$ as $h\to 0$. A similar computation making use of \eqref{eq:Derivative} guarantees continuity for $D^{\beta+e_j}f_\delta$ at $0$ and so our proof by induction is complete. Consequently, $f_\delta$ is smooth on $\mathbb{R}^d$ and, in view of \eqref{eq:Derivative}, for each multi-index $\beta$, there is a constant $C>0$ and an integer $l\geq 1$ for which
\begin{equation*}
\abs{(D^\beta f_\delta)(\xi)}\leq C(1+\abs{\xi}^l)
\end{equation*}
for all $\xi\in\mathbb{R}^d$. Since $\phi_\delta(\xi)=(P(\xi)+f_\delta(\xi))e^{-\abs{\xi}^{2m}}$ where $P(\xi)$ is a polynomial, it follows immediately from the above estimate that $\phi_\delta$ is a member of the Schwartz class, i.e., $\phi_\delta\in \mathcal{S}(\mathbb{R}^d)$. Using the fact that the (inverse) Fourier transform is an isomorphism of $\mathcal{S}(\mathbb{R}^d)$, we conclude that
\begin{equation*}
\widecheck{\phi_\delta}(x)=\frac{1}{(2\pi)^d}\int_{\mathbb{R}^d}\phi_{\delta}(\xi)e^{i\xi\cdot x}\,d\xi
\end{equation*}
defines a Borel measure $\mu_{1/\delta}$ via $d\mu_{1/\delta}=\widecheck{\phi_\delta}(x)\,dx$ having finite moments of all orders, c.f., \cite[Note 1.1.1]{StrombergBook}. In particular, $\mu_{1/\delta}\in\mathcal{M}_m$ and $\widehat{\mu_{1/\delta}}=\phi_\delta$.

It remains to verify that $\mu_{1/\delta}\in\mathcal{M}_m^M$.  To this end, let $\beta$ be a multi-index for which $\abs{\beta}\leq m$. In view of Leibniz' rule, we have
\begin{eqnarray}\label{eq:Leibniz}\nonumber
(D^\beta\phi_\delta)(\xi)&=&D^{\beta}\left(\left(P(\xi)+f_\delta(\xi)\right)e^{-\abs{\xi}^{2m}}\right)\\
&=&\sum_{\alpha\leq\beta} {\beta\choose\alpha} D^{\alpha}\left(P(\xi)+f_{\delta}(\xi)\right)D^{\beta-\alpha}\left(e^{-\abs{\xi}^{2m}}\right)
\end{eqnarray}
for $\xi\in\mathbb{R}^d$; here $\leq$ is the standard partial ordering on multi-indices, i.e., $\alpha\leq \beta$ provided $\alpha_j\leq\beta_j$ for $j=1,2,\dots,d$. Now, for each multi-index $\alpha\leq\beta$, it is easy to verify that
\begin{equation*}
D^{\beta-\alpha}\left(e^{-\abs{\xi}^{2m}}\right)\Big\vert_{\xi=0}=\begin{cases}
1 &\alpha=\beta\\
0 &\alpha<\beta
\end{cases}
\end{equation*}
in view of the fact that $\abs{\beta-\alpha}\leq m<2m$. With this observation and that fact that $(D^\alpha f_\delta)(0)=0$ for each $\alpha$, \eqref{eq:Leibniz} ensures that $\left(D^{\beta}\phi_{\delta}\right)(0)=\left(D^{\beta}P\right)(0)$ and so it follows that
\begin{equation*}
\left(D^{\beta}\phi_\delta\right)(0)=\sum_{\abs{\alpha}\leq m}\frac{(-i)^{\abs{\alpha}}M_\alpha}{\alpha!}D^\beta(\xi^\alpha)\Big\vert_{\xi=0}=(-i)^{\abs{\beta}}M_\beta.
\end{equation*}
Thus, by virtue of the previous lemma, we have
\begin{equation*}
\int_{\mathbb{R}^d}x^\beta\,d\mu_{1/\delta}(x)=(i)^\abs{\beta}(D^\beta\phi_\delta)(0)=M_\beta
\end{equation*}
for each $\beta\in\mathbb{N}^d$ for which $\abs{\beta}\leq m$ and therefore, $\mu_{1/\delta}\in\mathcal{M}_m^M$.
\end{proof}

\begin{lemma}\label{lem:Lipshitz}
For $\delta_1,\delta_2>0$, we have
\begin{equation*}
d_m(\mu_{1/\delta_1},\mu_{1/\delta_2})\leq \abs{\delta_2-\delta_2}.
\end{equation*}
\end{lemma}
\begin{proof}
Set $\rho=\abs{\delta_2-\delta_1}$. Observe that, for $\xi\neq 0$, we have
\begin{eqnarray*}
\frac{\abs{\phi_{\delta_1}(\xi)-\phi_{\delta_2}(\xi)}}{\abs{\xi}^m}&=&\abs{\xi}^2\abs{\sin(1/\abs{\xi}^2)}\abs{e^{-\delta_1/\abs{\xi}^2}-e^{\delta_2/\abs{\xi}^2}}e^{-\abs{\xi}^{2m}}\\
&\leq &\abs{\xi}^2\left(1-e^{-\rho/\abs{\xi}^2}\right).
\end{eqnarray*}
In view of the previous lemma, $\widehat{\mu_{1/\delta_k}}=\phi_{\delta_k}$ for $k=1,2$, and so it follows that
\begin{eqnarray*}
d_m(\mu_{1/\delta_1},\mu_{1/\delta_2})&=&\sup_{\xi\neq 0}\frac{\abs{\phi_{\delta_1}(\xi)-\phi_{\delta_2}(\xi)}}{\abs{\xi}^m}\\
&\leq &\sup_{\xi\neq 0}\abs{\xi}^2\left(1-e^{-\rho/\abs{\xi}^2}\right)\\
&=&\sup_{\theta>0} g(\theta)
\end{eqnarray*}
where $g(\theta)=\theta(1-e^{-\rho/\theta})$. By a straightforward exercise in single-variable calculus, it is easy to see that $g(\theta)\leq \rho$ for all $\theta>0$. Thus,
\begin{equation*}
d(\mu_{1/\delta_1},\mu_{1/\delta_2})\leq \sup_{\theta>0}g(\theta)\leq \rho=\abs{\delta_1-\delta_2},
\end{equation*}
as was asserted.
\end{proof}

\begin{lemma}\label{lem:NotSmooth}
For $\delta=0$, $\phi_\delta=\phi_0$ is not twice differentiable and hence, for any $m\geq 2$, $\phi_0\notin C^m(\mathbb{R}^d)$.
\end{lemma}
\begin{proof}
Consider $f:\mathbb{R}\to\mathbb{R}$ defined by
\begin{equation*}
f(x)=e^{|xe_1|^{2m}}\phi_0(xe_1)-P(xe_1)=\abs{x}^{m+2}\sin(1/|x|^{m})
\end{equation*}
for $x\in\mathbb{R}$. It is straightforward to verify that $f$ is continuously differentiable on $\mathbb{R}$ with 
\begin{equation*}
f'(x)=\begin{cases}x\Big(
(m+2)|x|^{m}\sin(1/\abs{x}^m)-m\cos(1/\abs{x}^m)\Big)& x\neq 0\\
0 & x=0.
\end{cases}
\end{equation*}
With this, we see that the limit
\begin{equation*}
\lim_{h\to 0}\frac{f'(h)-f'(0)}{h}=\lim_{h\to 0}\left((m+2)\abs{h}^m\sin(1/\abs{h}^m)-m\cos(1/\abs{h}^m)\right)
\end{equation*}
does not exist and hence $f$ is not twice differentiable at $0$. Of course, if $\phi_0\in C^m(\mathbb{R}^d)\subseteq C^2(\mathbb{R}^d)$, the construction of $f$ would imply that $f\in C^2(\mathbb{R})$ which is not the case. Hence $\phi_0\notin C^m(\mathbb{R}^d)$.
\end{proof}
\begin{proof}[Proof Theorem \ref{thm:Metric_Not_Compete}]
We first verify that $d_m$ is finite on $\mathcal{M}_m^M$. To this end, let $\nu_1,\nu_2\in\mathcal{M}_m^M$. In view of Lemma \ref{lem:MomentsAndSmoothness}, an application of Taylor's theorem (with remainder in Peano form) guarantees that, for $k=1,2$,
\begin{eqnarray*}
\widehat{\nu_k}(\xi)&=&\sum_{\abs{\beta}\leq m}\frac{D^\beta(\widehat{\nu_k})(0)}{\beta!}\xi^\beta+\mathcal{E}_k(\xi)\abs{\xi}^m\\
&=&\sum_{\abs{\beta}\leq m}\frac{(-i)^{\abs{\beta}}M_\beta}{\beta!}\xi^\beta +\mathcal{E}_k(\xi)\abs{\xi}^m
\end{eqnarray*}
where $\mathcal{E}_k$ is a continuous function with $\mathcal{E}_k(\xi)\to 0$ as $\xi\to 0$; here, we have used the fact that $\nu_1$ and $\nu_2$ have matching moments (given by $M=\{M_\beta\}$) up to order $m$). Thus
\begin{equation*}
\sup_{0<\abs{\xi}\leq 1}\frac{\abs{\widehat{\nu_1}(\xi)-\widehat{\nu_2}(\xi)}}{\abs{\xi}^m}=\sup_{0<\abs{\xi}\leq 1}\abs{\mathcal{E}_1(\xi)-\mathcal{E}_2(\xi)}=:C<\infty.
\end{equation*}
Consequently,
\begin{eqnarray*}
\lefteqn{d(\nu_1,\nu_2)=\sup_{\xi\neq 0}\frac{\abs{\widehat{\nu_1}(\xi)-\widehat{\nu_2}(\xi)}}{\abs{\xi}^m}}\\
&&\leq C+\sup_{\abs{\xi}> 1}\abs{\widehat{\nu_1}(\xi)-\widehat{\nu_2}(\xi)}\leq C+\abs{\nu_1}(\mathbb{R}^d)+\abs{\nu_2}(\mathbb{R}^d)<\infty
\end{eqnarray*}
where we have used that, for $k=1,2$, the total variation $\abs{\nu_k}$ of $\nu_k$ is a finite measure. Thus $d_m$ is finite for each pair $\nu_1,\nu_2\in\mathcal{M}_m^M$ and, since $\mathcal{M}_m^M$ is non-empty by  Lemma \ref{lem:SchwartzFunctions}, we may conclude that $(\mathcal{M}_m^M,d_m)$ is a metric space. 

It remains to show that $(\mathcal{M}_m^M,d_m)$ is not complete. To this end, for each positive integer $j$, define $\mu_j\in\mathcal{M}_m^M$ to be the measure with Fourier transform $\phi_{1/j}$ guaranteed by Lemma \ref{lem:SchwartzFunctions}. By virtue of the Lemma \ref{lem:Lipshitz}, 
\begin{equation*}
d(\mu_j,\mu_k)\leq \abs{\frac{1}{j}-\frac{1}{k}}
\end{equation*}
and thus $\{\mu_j\}\subseteq\mathcal{M}_m^M$ is a Cauchy sequence in the metric $d_m$. For suppose that, for some measure $\mu\in\mathcal{M}_m^M$, $\lim_{j\to\infty}d_m(\mu_j,\mu)=0$. By the definition of $d_m$, it follows that $\widehat{\mu}(\xi)=\lim_{j\to\infty}\phi_{1/j}(\xi)=\phi_0(\xi)$ for each $\xi\in\mathbb{R}^d$. In view of Lemma \ref{lem:MomentsAndSmoothness}, we conclude that $\phi_0\in C^m(\mathbb{R}^d)$ but this stands in contradiction to Lemma \ref{lem:NotSmooth}. Thus, $(\mathcal{M}_m^M,d_m)$ is not complete.
\end{proof}
The example in the above proof can be used to illustrate that the conclusion of Lemma 3.2 of \cite{Sta20} fails in the context of complex measures. Through a careful study of Section 2.9 of Chapter VI of \cite{K04}, it becomes clear that the validity of Lemma 3.2 of \cite{Sta20} is tied to the remarkable nature of positive definition functions (arising as the Fourier transforms of positive measures via Bochner's theorem). 

\end{document}